\documentclass[a4paper]{amsart}
\usepackage[utf8]{inputenc}
\usepackage{graphicx,tikz}
\usepackage{tikz-cd}
\usetikzlibrary{calc}
\usepackage{amssymb}

\theoremstyle{plain}
  \newtheorem{thm}{Theorem}
  \newtheorem{defn}{Definition}
  \newtheorem{prop}{Proposition}

\theoremstyle{definition}
  \newtheorem{example}{Example}
  \newtheorem*{rem}{Remark}

\newcommand{\mf}{\mathfrak}
\newcommand{\on}{\operatorname}

\newcommand{\g}{\mathfrak{g}}
\newcommand{\h}{\mathfrak{h}}
\newcommand{\dd}{\mathfrak{d}}
\newcommand{\Hom}{\operatorname{Hom}}
\newcommand{\la}{\langle}
\newcommand{\ra}{\rangle}
\newcommand{\R}{\mathbb{R}}

\title[Left and right centers in quasi-Poisson geometry]{Left and right centers in quasi-Poisson geometry of moduli spaces}
\author{Pavol \v{S}evera}
\address{Section of Mathematics, University of Geneva, Switzerland}
\email{pavol.severa@gmail.com}
\thanks{Supported in part by  the grant MODFLAT of the European Research Council and the NCCR SwissMAP of the Swiss National Science Foundation.}
\begin{document}
\maketitle

\begin{abstract}
We introduce left central and right central functions and left and right leaves in quasi-Poisson geometry, generalizing central (or Casimir) functions and symplectic leaves from Poisson geometry. They lead to a new type of (quasi-)Poisson reduction, which is both simpler and more general than known quasi-Hamiltonian reductions.  We study these notions in detail for moduli spaces of flat connections on surfaces, where the quasi-Poisson structure is given by an intersection pairing on homology.

\end{abstract}

\section{Introduction}
A function on a Poisson manifold is \emph{central} (or \emph{Casimir}) if it Poisson-commutes with every function. A function is central iff it is constant on each \emph{symplectic leaf} of the Poisson manifold. 

Among the most interesting Poisson spaces are moduli spaces of flat connections on an oriented compact surface, with the Poisson structure of Atiyah-Bott \cite{a-b} and Goldman \cite{gold} given by an intersection pairing on the surface. Their symplectic leaves are obtained by fixing the conjugacy classes of the holonomies along the boundary circles.

Moduli spaces on surfaces with marked points on the boundary carry a \emph{quasi-Poisson structure}
 \cite{a-m-m,a-ks-m}. A quasi-Poisson manifold is by definition a manifold with an action of a Lie algebra $\g$ and with an invariant bivector field $\pi$, satisfying
$$[\pi,\pi]/2=\rho(\phi),$$
where $\rho(\phi)$ is a 3-vector field coming from an invariant inner product on $\g$ and from the structure constants of $\g$. These moduli spaces have the advantage of being smooth and can be built out of simple pieces using the operation of fusion. Moduli spaces without marked points can then be obtained via a quasi-Hamiltonian reduction.

In the same way as Poisson structures are analogous to non-commutative algebras, quasi-Poisson structures are analogous to non-commutative algebras in a braided monoidal category. Motivated by this analogy, we introduce \emph{left central} and \emph{right central} functions, and the corresponding \emph{left} and \emph{right leaves} of quasi-Poisson manifolds. Besides the interesting geometry of these foliations, these notions bring a new type of  reduction, the \emph{central reduction} of quasi-Poisson manifolds, which produces symplectic and Poisson manifolds.

Despite the simplicity of central reduction (it is, in a way, simpler than the standard moment map reduction of symplectic manifolds), it contains as a special case the so far most general quasi-Hamiltonian reduction of \cite{lb-se-big}. Replacing moment maps with left/right centers also allows us to define the fusion of $D/H$-valued moment maps of \cite{a-ks}, which was so far missing.

As we mentioned above, the motivating and also the most important examples of quasi-Poisson manifolds are moduli spaces of flat $\g$-connections on a surface with marked points on the boundary. We reformulate these quasi-Poisson structures as an intersection pairing of homology with coefficients in a local system. This formulation is manifestly natural (it doesn't use any splitting of the surface to simpler pieces) and also very similar the formulation of Atiyah-Bott and Goldman. It also reduces all non-degeneracy problems to Poincar\'e duality, and left/right centers are found simply as holonomies along the parts of the boundary that don't intersect any cycle. Central reduction then produces  interesting examples of symplectic and Poisson manifolds.


\subsection*{Acknowledgment}
I would like to thank to David Li-Bland for many useful discussions.

\section{Leaves and central maps in quasi-Poisson geometry}

Let $\g$ be a Lie algebra with an invariant element $t\in(S^2\g)^\g$. Let $\phi\in\bigwedge^3\g$ be given by
$\phi=-[t^{1,2},t^{2,3}]/4$, i.e.
$$\phi(\alpha,\beta,\gamma)=\frac{1}{4}\bigl\langle[t^\sharp\alpha,t^\sharp\beta],\gamma\bigr\rangle\text{ for all }\alpha,\beta,\gamma\in\g^*.$$

The following definition is due to Alekseev, Kosmann-Schwarzbach, and Meinrenken \cite{a-ks-m}.
\begin{defn}\label{def:qpoiss}
A \emph{$\g$-quasi-Poisson manifold} is a manifold $M$ with an action $\rho$ of $\g$ and with a $\g$-invariant bivector field $\pi$, satisfying
$$[\pi,\pi]/2=\rho^{\otimes 3}(\phi).$$
A map $F:M\to M'$ between two $\g$-quasi-Poisson manifolds is \emph{quasi-Poisson} if it is $\g$-equivariant and if $F_*\pi_M=\pi_{M'}$.
\end{defn}

If $(M,\rho,\pi)$ is a quasi-Poisson manifold, let
$$\sigma=\pi+\frac{1}{2}\,\rho^{\otimes 2}(t)\in\Gamma(T^{\otimes2}M).$$
\begin{defn}
A function $f\in C^\infty(M)$ is \emph{left-central} if $\sigma(df,\cdot)=0$ and \emph{right-central} if $\sigma(\cdot,df)=0$.
\end{defn}

\begin{rem}
This definition, as well as many other things in this paper, is motivated by the following quantum analogue. Let $\Phi$ be a Drinfeld associator and let $U\g\text{-mod}^\Phi$ be the category of $U\g$-modules with the braiding and the associativity constraint defomed by $t$ and $\Phi$ (see \cite{drinfeld} for details). Suppose that $\g$ acts on $M$, and that $\ast$ is a star product on $M$ making $C^\infty(M)$ to an associative algebra in $U\g\text{-mod}^\Phi$, i.e.
$$\ast\circ(\ast\otimes1)=\ast\circ(1\otimes\ast)\circ(\Phi\,\cdot).$$
Then, as observed in \cite{en-et}, there is a $\g$-quasi-Poisson structure on $M$ given by
$$\pi(df,dg)=(f\ast g-g\ast f)/\hbar\mod\hbar.$$
The $\hbar$-term of the braided commutator
$$
\begin{tikzpicture}[baseline=1cm]
\coordinate(E) at (0,2.5);
\coordinate[label=right:{$\ast$}](P) at (0,1.7);
\node(f) at (-0.5,0) {$f$};
\node(g) at (0.5,0) {$\vphantom{f}g$};
\draw (g)..controls+(0,1) and +(0.5,-0.5)..(P);
\draw (f)..controls+(0,1) and +(-0.5,-0.5)..(P);
\draw (P)--(E);
\end{tikzpicture}
\,-
\begin{tikzpicture}[baseline=1cm]
\coordinate(E) at (0,2.5);
\coordinate[label=right:{$\ast$}](P) at (0,1.7);
\node(f) at (-0.5,0) {$f$};
\node(g) at (0.5,0) {$\vphantom{f}g$};
\draw[line width=1ex,white](f)..controls+(0,1) and +(1,-1)..($(P)+(0.02,-0.02)$);
\draw (f)..controls+(0,1) and +(1,-1)..(P);
\draw[line width=1ex,white] (g)..controls+(0,1) and +(-1,-1)..($(P)+(-0.02,-0.02)$);
\draw (g)..controls+(0,1) and +(-1,-1)..(P);
\draw (P)--(E);
\end{tikzpicture}
$$
is then equal to $\sigma(df,dg)$; the braiding contributes the  symmetric part  $\frac{1}{2}\,\rho^{\otimes 2}(t)$ of $\sigma$.

As a general rule, we shall only use those parts of quasi-Poisson geometry which have a clear quantum analogues, i.e.\ which make sense for algebras in the braided monoidal category $U\g\text{-mod}^\Phi$.

Let us also remark here that one can define quasi-Poisson algebras in an arbitrary infinitesimally braided category $\mathcal C$, i.e.\ a linear symmetric monoidal category with a natural transformation $t^{X,Y}:X\otimes Y\to X\otimes Y$ satisfying the relations $t^{X,Y}=t^{Y,X}$ and $t^{X\otimes Y,Z}=t^{X,Z}+t^{Y,Z}$: a quasi-Poisson algebra $A$ in $\mathcal C$ is by definition a commutative algebra together with a skew-symmetric morphism $\{,\}:A\otimes A\to A$ satisfying the Leibniz rule and the quasi-Jacobi identity 
$$\{\{.,.\},.\}+c.p.=-\frac{1}{4}\bigl[t^{A,A}\otimes\on{id},\on{id}\otimes t^{A,A}\bigr]:A\otimes A\otimes A\to A.$$
\end{rem}

\begin{defn}
A $\g$-quasi-Poisson manifold is \emph{quasi-Poisson-commutative} if $\sigma=0$. Equivalently, it is a $\g$-manifold such that the stabilizers of points are coisotropic Lie subalgebras of $\g$, together with the bivector field $\pi=0$.
\end{defn}
Here a Lie subalgebra $\mf c\subset\g$ is called \emph{coisotropic} if the image of $t$ in $S^2(\g/\mf c)$ vanishes.
The most straightforward example of a quasi-Poisson-commutative manifold is $G/C$, where $G$ is a Lie group with the Lie algebra $\g$ and $C\subset G$ is a closed subgroup with a coisotropic Lie algebra $\mf c\subset\g$.

\begin{defn}
If $M$ is $\g$-quasi-Poisson and $N$ is $\g$-quasi-Poisson-commutative, a map $F:M\to N$ is \emph{left (or right) central}, if $F$ is $\g$-equivariant and if $F^* f\in C^\infty(M)$ is left (or right) central for every $f\in C^\infty(N)$.
\end{defn}
Notice that a left (or right) central map is automatically quasi-Poisson.

We can characterize left (or right) central maps as follows.
 There are three natural integrable distributions on $M$:  
$$T^LM=\text{the image of }a_L:T^*M\to TM,\ \alpha\mapsto\sigma(\cdot,\alpha),$$
$$T^RM=\text{the image of }a_R:T^*M\to TM,\ \alpha\mapsto\sigma(\alpha,\cdot),$$
$$T^{big}M=\text{the image of }a:\g\oplus T^*M\to TM,\ (u,\alpha)\mapsto\rho(u)+\pi(\alpha,\cdot).$$
Their integrability follows from the fact that $a_L$, $a_R$, and $a$ are the anchor maps for certain Lie algebroid structures (see \cite{qhg} for $a_L$ and $a_R$ and \cite{b-c} for $a$). Notice that 
\begin{equation*}
T^LM+\rho(\g)=T^RM+\rho(\g)=T^{big}M
\end{equation*}
and that $\sigma$ gives a non-degenerate pairing 
$$\sigma^{-1}:T^LM\times T^RM\to\R$$
given by 
$$\sigma^{-1}(u,v)=\sigma(\alpha,\beta),\quad\alpha,\beta\in T^*M\text{ such that }v=\sigma(\alpha,\cdot),u=\sigma(\cdot,\beta).$$
This pairing is a generalization of the symplectic form on the symplectic leaves of a Poisson manifold: if $\g=0$ then $\sigma=\pi$ and $T^LM=T^RM$ is the tangent space of the symplectic leaves of the Poisson structure $\pi$, and $\sigma^{-1}$ is the symplectic form on $T^LM=T^RM$.

\begin{defn}
The integral leaves of $T^LM$ are the \emph{left leaves} of $M$, the integral leaves of $T^RM$ are the \emph{right leaves} of $M$, and the integral leaves of $T^{big}M$ are the \emph{big leaves} of $M$.
\end{defn}

Notice  that a function is left-central (right-central) iff it is constant on the left (right) leaves.  An equivariant map $F:M\to N$ is thus left (or right) central iff it is constant on each left (or right) leaf.

Since $\sigma$ is $\g$-invariant, each of the three foliations is $\g$-invariant. Big leaves are (minimal) $\g$-quasi-Poisson submanifolds of $M$. Left and right leaves are contained within the big leaves; they are (in general) not quasi-Poisson submanifolds (see, however, Theorem \ref{thm:qp-qp}). Let us also notice that while foliation by big leaves can be singular (the dimension of big leaves can jump), the foliation of a given big leaf by left (or right) leaves is a true foliation:

\begin{prop}
The action of $\g$ on the space of left (or right) leaves contained within a single big leaf $Y\subset M$ is locally transitive; in particular, left leaves form a foliation of $Y$. The stabilizer of any left (or right) leaf is a coisotropic subalgebra.
\end{prop}
\begin{proof}
Transitivity follows from $T^LM+\rho(\g)=T^{big}M$. The local space of left (or right) leaves within $Y$ is quasi-Poisson-commutative, so the stabilizers must be coisotropic. 
\end{proof}

Let us now single out quasi-Poisson manifolds with the most interesting left and right centers.

\begin{defn}
A $\g$-quasi-Poisson manifold $M$ is \emph{split-symplectic} if $t\in S^2\g$ is non-degenerate, if $T^{big}M=TM$ (i.e.\ $M$ has a single big leaf), and if the stabilizers of the left (or equivalently right) leaves are Lagrangian.
\end{defn}
Here we call a Lie subalgebra $\mf c\subset\g$ \emph{Lagrangian} (provided $t$ is non-degenerate, so that it defines a pairing on $\g$) if $\mf c^\perp=\mf c$. It implies that $\dim\g$ is even, as $\dim\g=2\dim\mf c$. The split-symplectic condition says that the dimension of the left (and thus also of the right) leaves is as small as possible (supposing $T^{big}M=TM$ and non-degenerate $t$), namely
$$\on{rank}T^LM=\on{rank}T^RM=\dim M-\frac{1}{2}\dim\g.$$

\begin{example}\label{ex:main}
Suppose that $t_\g\in (S^2\g)^\g$ is non-degenerate. Let $\bar\g$ be $\g$ with $t_{\bar\g}=-t_\g$, and let $\dd=\g\oplus\bar\g$, with $t_\dd=t_\g\oplus t_{\bar\g}$.

If $G$ is a connected group integrating $\g$, then $G$ with the action of $\dd$, $(u,v)\mapsto u^L-v^R$, is a $\dd$-quasi-Poisson-commutative manifold (here $u^L$ and $v^R$ are the left and right invariant vector fields on $G$ corresponding to $u,v\in\g$).

Let us now consider $G\times G$ with the diagonal action of $\dd$. The bivector field
$$\pi=\sum_{ij}t^{ij}\bigl((0,e_i^L)\wedge(e_j^L,0)-(0,e_i^R)\wedge(e_j^R,0)\bigr)$$
($e_i$ is a basis of $\g$ and $t_\g=\sum_{ij}t^{ij}\,e_i\otimes e_j$) makes it to a $\dd$-quasi-Poisson manifold. As
$$\sigma=\sum_{ij}t^{ij}\bigl((0,e_i^L)\otimes(e_j^L,0)-(0,e_i^R)\otimes(e_j^R,0)\bigr),$$
the projection $p_1:G\times G\to G$ to the first factor is left-central and the projection $p_2$ to the second factor is right-central.

The big leaves in $G\times G$ are $m^{-1}(\mathcal O)$, where $m:G\times G\to G,(g_1,g_2)\mapsto g_1g_2^{-1}$, and $\mathcal O\subset G$ runs over conjugacy classes of $G$. For every conjugacy class $\mathcal O\subset G$ the manifold $m^{-1}(\mathcal O)$ is split-symplectic, and its left (resp.\ right) leaves are the fibres of $p_1$ (resp. $p_2$).

The $\dd$-quasi-Poisson manifold $G\times G$ can be described as the fusion product $G\circledast G$ (see Section \ref{sec:fusion}). As we shall see in Section \ref{sec:moduli}, it can also be interpreted as the moduli space of flat $G$-bundles on the annulus with two marked points on the exterior circle:
$$
\begin{tikzpicture}[>=stealth]
\fill[color=white!90!black] (0,0)circle [radius=1cm];
\draw (0,0) circle [radius=1cm];
\fill[color=white] (0,0)circle [radius=0.3cm];
\draw (0,0) circle [radius=0.3cm];
\fill[color=black] (0,1)circle [radius=0.05cm];
\fill[color=black] (0,-1)circle [radius=0.05cm];
\node[above] at (0,1) {$-$};
\node[below] at (0,-1) {$+$};
\node[left] at (-1,0) {$g_1$};
\node[right] at (1,0) {$g_2$};
\draw[->] (1,0) arc [start angle=180, end angle=181, radius=1cm];
\draw[->] (-1,0) arc [start angle=180, end angle=181, radius=1cm];
\end{tikzpicture}
$$
 The submanifold $m^{-1}(\mathcal O)$  is given by restricting the holonomy along the internal circle to be in $\mathcal O$.
\end{example}

\section{Central reduction of quasi-Poisson manifolds}
If  $\mf c\subset \g$ is a coisotropic Lie subalgebra, i.e.\ if the image of $t\in S^2\g$ in $S^2(\g/\mf c)$ vanishes, then also the image of $\phi\in\bigwedge^3\g$ in $\bigwedge^3(\g/\mf c)$ vanishes. As a consequence, we get the following result.

\begin{prop}[\cite{lb-se}]\label{prop:red}
If $M$ is a $\g$-quasi-Poisson manifold then the algebra of $\mf c$-invariant functions $C^\infty(M)^{\mf c}\subset C^\infty(M)$ is a Poisson algebra, with the Poisson bracket
$$\{f,g\}=\pi(df,dg)=\sigma(df,dg)$$
inherited from $C^\infty(M)$. In particular, if $M/\mf c$ is a manifold then it is Poisson.
\end{prop}

\begin{rem}
By $M/\mf c$ being a manifold we mean the following: there is a manifold $M'$ and a surjective submersion $M\to M'$ such that its fibres are the $\mf c$-orbits on $M$.
\end{rem}

The following observation is obvious, but also central for this section.
\begin{prop}\label{prop:cr}
Let $M$ be a $\g$-quasi-Poisson manifold and $f\in C^\infty(M)$ a left (or right) central function. If $f$ is $\mf c$-invariant for some coisotropic $\mf c\subset\g$ then it is central (i.e.\ Casimir) in the Poisson algebra $C^\infty(M)^{\mf c}$.
\end{prop}
\begin{proof}
If $g\in C^\infty(M)^{\mf c}$ then $\{f,g\}=\sigma(df,dg)=0$, as $f$ is left central. The proof for right-central $f$'s is similar.
\end{proof}

Let us call a diagram
$$
\begin{tikzcd}[column sep=small]
& M \arrow{dl}[swap]{\mu_L}\arrow{dr}{\mu_R} & \\
N_L & & N_R
\end{tikzcd}
$$
a \emph{central pair} if $M$ is a $\g$-quasi-Poisson manifold and $\mu_L$ and $\mu_R$ are a left and a right central map respectively.

As a version of Proposition \ref{prop:cr} we get the following reduction method.
\begin{thm}[Central reduction]\label{thm:cred}
Let 
\begin{equation}\label{eq:cpair}
\begin{tikzcd}[column sep=small]
& M \arrow{dl}[swap]{\mu_L}\arrow{dr}{\mu_R} & \\
N_L & & N_R
\end{tikzcd}
\end{equation}
be a central pair. Let $\mf c\subset\g$ be a coisotropic Lie subalgebra and let $\mathcal O_L\subset N_L$ and $\mathcal O_R\subset N_R$ be $\mf c$-invariant submanifolds. If $\mu_L\times\mu_R$ is transverse to $\mathcal O_L\times \mathcal O_R$ and if $M/\mf c$ is a manifold then
\begin{equation}\label{eq:subP}
\bigl(\mu_L^{-1}(\mathcal O_L)\cap\mu_R^{-1}(\mathcal O_R)\bigr)/\mf c\subset M/\mf c
\end{equation}
is a Poisson submanifold.
\end{thm}
\begin{proof}
The Poisson bivector field $\pi_{red}$ on $M/\mf c$ is given by 
$$\pi_{red}(\alpha,\beta)=\pi(\alpha,\beta)=\sigma(\alpha,\beta)$$
 for any $\alpha,\beta\in T^*_xM$ which descend to $T^*_{[x]}(M/\mf c)$, i.e.\ which are in the kernel of $i_{\rho(u)}$ for every $u\in\mf c$. If $\alpha$ is the pull-back of a covector from $N_L$ then by left-centrality of $\mu_L$ we have $\sigma(\alpha,\cdot)=0$. This shows that $\pi_{red}(\alpha,\cdot)=0$ whenever $\alpha$ is the pull-back of an element of the conormal bundle of $\mathcal O_L$. A similar argument applies to $\mathcal O_R$, hence \eqref{eq:subP} is indeed a Poisson submanifold.
\end{proof}

A natural question is how to describe the symplectic leaves of the Poisson manifold $M/\mf c$. Theorem \ref{thm:cred} does it under the following circumstances.

\begin{defn}
A central pair \eqref{eq:cpair} is \emph{split-symplectic} if $t\in S^2\g$ is non-degenerate, $T^{big}M=TM$ (i.e.\ $M$ has just one big leaf), and the actions of $\g$ on  $N_L$ and $N_R$ are transitive with Lagrangian stabilizers.
\end{defn}
\begin{rem}
If \eqref{eq:cpair} is a split-symplectic central pair then $M$ is a split-symplectic $\g$-quasi-Poisson manifold, and the fibers of $\mu_L$ ($\mu_R$) are the left (right) leaves of $M$. If, on the other hand, $M$ is a split-symplectic $\g$-quasi-Poisson manifold then we can set $N_L$ ($N_R$) to be the local space of left (right) leaves of $M$, and get (locally) a split-symplectic central pair. Split-symplectic central pairs and split-symplectic quasi-Poisson manifolds are thus essentially the same thing.
\end{rem}

\begin{thm}[Split-symplectic central reduction]\label{thm:scred}
Suppose, in the context of Theorem~\ref{thm:cred}, that \eqref{eq:cpair} is split-symplectic, $\mf c\subset\g$ is Lagrangian, and  $\mathcal O_L$ and $\mathcal O_R$ are $\mf c$-orbits. Then
$$\bigl(\mu_L^{-1}(\mathcal O_L)\cap\mu_R^{-1}(\mathcal O_R)\bigr)/\mf c\subset M/\mf c$$
is a symplectic leaf.
\end{thm}
\begin{proof}
It follows from Proposition \ref{prop:ndeg} in Appendix, where we set $V=T^*_xM$,  $W=\g^*$, $f=\rho^*$, and $C=\on{Ann}\mf c$.
\end{proof}

\begin{example}
Let $\mathcal O\subset G$ be a conjugacy class. Let us consider the split-symplectic central pair
$$
\begin{tikzcd}[column sep=small]
& m^{-1}(\mathcal O) \arrow{dl}[swap]{p_1}\arrow{dr}{p_2} & \\
G & & G
\end{tikzcd}
$$
defined in Example \ref{ex:main}. Notice that $\g_{diag}\subset\g\oplus\bar \g$ is a Lagrangian Lie subalgebra. Let us choose a $\g_{diag}$-orbit, i.e.\ a conjugacy class, in each $G$; we shall denote them $\mathcal O_L$ and $\mathcal O_R$ as above. Central reduction gives us a symplectic form on (the non-singular part of)
$$\{(g_1,g_2,g_3)\in\mathcal O_L\times\mathcal O_R\times\mathcal O;\,g_1=g_3g_2\}/G$$
where $G$ acts by conjugation. This symplectic space can be identified with the moduli space of flat $G$-bundles over a sphere with 3 punctures, with the holonomies around the punctures in $\mathcal O_L$, $\mathcal O_R$, $\mathcal O$.

\end{example}

\subsection*{Partial reduction}
The reduction by coisotropic Lie algebra we described above is a special case of a reduction of quasi-Poisson manifolds to quasi-Poisson manifolds. 

For any Lie subalgebra $\mf c\subset\g$ let
$$\mf c^\perp=t(\on{Ann}\mf c,\cdot)\subset\g.$$
The Lie subalgebra $\mf c$ is coisotropic iff $\mf c^\perp\subset\mf c$; in that case $\mf c^\perp\subset\mf c$ is an ideal. The element $t\in S^2\g$ descends  to an element $t'\in S^2(\mf c/\mf c^\perp)\subset S^2(\g/\mf c^\perp)$ via the projection $S^2\g\to S^2(\g/\mf c^\perp)$. The element $\phi\in\bigwedge^3\g$ descend in the same way to the element $\phi'\in\bigwedge^3(\mf c/\mf c^\perp)$ corresponding to $t'$. The same applies if we replace $\mf c/\mf c^\perp$ with $\mf c/\h$ for any ideal $\h\subset\mf c $ containing $\mf c^\perp$.
As a result we get the following.
\begin{prop}[\cite{lb-se}]\label{prop:pred}
If $M$ is a $\g$-quasi-Poisson manifold, $\mf c\subset\g$ a coisotropic Lie subalgebra, and $\h\subset\mf c$ an ideal such that $\mf c^\perp\subset\h$, and if $M/\h$ is a manifold, then $M/\h$ is $\mf c/\h$-quasi-Poisson, with the bivector field pushed-forward from $M$.
\end{prop}

A useful particular case is when $\g$ is a direct sum $\g=\g_1\oplus\g_2$ with $t=t_1+t_2$, $t_i\in(S^2\g_i)^{\g_i}$, and $\mf c_1\subset\g_1$ is a coisotropic Lie subalgebra. If we set $\h =\mf c_1$ and $\mf c=\mf c_1\oplus\g_2$ then  Proposition \ref{prop:pred}
 says that $M/\mf c_1$ is $\g_2$-quasi-Poisson.

Theorems \ref{thm:cred} and \ref{thm:scred} have now the following extensions (we omit the proofs, as they are very similar).

\begin{thm}[Central reduction 2]\label{thm:cred2}
Suppose that \eqref{eq:cpair} is a central pair. Let $\mf c\subset\g$ be a coisotropic Lie subalgebra and $\h\subset\mf c$ an ideal containing $\mf c^\perp$. Let $\mathcal O_L\subset N_L$ and $\mathcal O_R\subset N_R$ be $\mf c$-invariant submanifolds. If $\mu_L\times\mu_R$ is transverse to $\mathcal O_L\times \mathcal O_R$ and if $M/\h$, $\mathcal O_L/\h$ and $\mathcal O_R/\h$ are manifolds then
$$M':=\bigl(\mu_L^{-1}(\mathcal O_L)\cap\mu_R^{-1}(\mathcal O_R)\bigr)/\h\subset M/\h$$
is a $\mf c/\h$-quasi-Poisson submanifold, and
$$
\begin{tikzcd}[column sep=small]
& M' \arrow{dl}\arrow{dr} & \\
\mathcal O_L/\h & & \mathcal O_R/\h
\end{tikzcd}
$$
is a central pair.
\end{thm}

\begin{thm}[Split-symplectic central reduction 2]\label{thm:sscr2}
If, in the context of Theorem \ref{thm:cred2}, the central pair \eqref{eq:cpair} is split symplectic,  if $\h=\mf c^\perp$, and if $\mathcal O_L$ and $\mathcal O_R$ are $\mf c$-orbits, then
$$
\begin{tikzcd}[column sep=small]
& M' \arrow{dl}\arrow{dr} & \\
\mathcal O_L/\mf c^\perp & & \mathcal O_R/\mf c^\perp
\end{tikzcd}
$$
is a split-symplectic central pair.
\end{thm}

\begin{example}
Let us consider again the split-symplectic $\g\oplus\bar\g$-central pair
$$
\begin{tikzcd}[column sep=small]
& m^{-1}(\mathcal O) \arrow{dl}[swap]{p_1}\arrow{dr}{p_2} & \\
G & & G
\end{tikzcd}
$$
defined in Example \ref{ex:main}. Suppose that $\g$ is semisimple and that  $\mf b\subset\g$ is a Borel subalgebra. We shall reduce this central pair by the coisotropic subalgebra $\mf c:=\mf b\oplus\mf b\subset\g\oplus\bar\g$. Notice that $\mf c^\perp=\mf n\oplus\mf n$, where $\mf n\subset\mf b$ is the nilpotent radical, and that $\mf c/\mf c^\perp=\mf t\oplus\bar{\mf t}$, where $\mf t\subset\g$ is the Cartan subalgebra.

Let us observe that $\mf c$-orbits in $G$ give us the Bruhat decomposition of $G$: for any element $w$ of the Weyl group $W_G$ we have the orbit $\mathcal O_w=BwB\subset G$, where $B$ is the Borel subgroup. 

We can identify the quotient $\mathcal O_w/\mf c^\perp$ with $Tw\subset G$, where $T$ is the Cartan subgroup. As a result, for any pair $w_1,w_2\in W_G$ Theorem \ref{thm:sscr2} gives us a split-symplectic $\mf t\oplus\bar{\mf t}$-central pair
$$
\begin{tikzcd}[column sep=small]
& M_{w_1,w_2} \arrow{dl}[swap]{p_1}\arrow{dr}{p_2} & \\
Tw_1 & & Tw_2
\end{tikzcd}
$$
As $\mf t\oplus\bar{\mf t}$ is Abelian, the quasi-Poisson structure on $M_{w_1,w_2}$ is actually Poisson.
\end{example}

\section{Fusion}\label{sec:fusion}
If $M_1$ and $M_2$ are $\g$-quasi-Poisson manifolds then their product $M_1\times M_2$ (with $\pi_1+\pi_2$) is $\g\oplus\g$-quasi-Poisson. To make it $\g$-quasi-Poisson, where  $\g\hookrightarrow\g\oplus\g$ is the diagonal inclusion, one needs to use a twist (the reason is that $\g_\text{diag}\subset\g\oplus\g$ is not a quasi-Lie sub-bialgebra, but becomes so after the twist). The result is as follows. 
\begin{defn}[\cite{a-ks-m}]
The \emph{fusion product} $M_1\circledast M_2$ of two $\g$-quasi-Poisson manifolds $(M_1,\rho_1,\pi_1)$ and $(M_2,\rho_2,\pi_2)$ is the manifold $M_1\times M_2$ with the $\g$-quasi-Poisson structure given by the diagonal action of $\g$ and by the bivector field
\begin{equation}\label{eq:fusion-pi}
\pi_\circledast=\pi_1+\pi_2-\frac{1}{2}(\rho_1\wedge\rho_2)(t).
\end{equation}
\end{defn}

Let us notice that the corresponding tensor field $\sigma_\circledast$ on $M_1\circledast M_2$ is given by
\begin{equation}\label{eq:fusion-si}
\sigma_\circledast=\sigma_1+\sigma_2+(\rho_2\otimes\rho_1)(t).
\end{equation}

The fusion product makes the category of $\g$-quasi-Poisson manifolds to a (non-braided) monoidal category. 

\begin{rem}
Monoids in a braided monoidal category form a (non-braided) monoidal category:  if $A$ and $B$ are monoids  then $A\otimes B$ is a monoid as well, with the product
$$(A\otimes B)\otimes(A\otimes B)\to A\otimes B$$
given by the diagram
$$
\begin{tikzpicture}[baseline=1cm]
\coordinate (diff) at (0.5,0);
\coordinate (dy) at (0,0.5);
\node(A1) at (0,0) {$A$};
\node(B1) at ($(A1)+(diff)$) {$B$};
\node(A2) at (2,0) {$A$};
\node(B2) at ($(A2)+(diff)$) {$B$};
\node(A3) at ($(A1)+(0,2.5)+0.5*(diff)$) {$A$};
\node(B3) at ($(A2)+(0,2.5)+0.5*(diff)$) {$B$};
\coordinate(A) at ($(A3)-(dy)$);
\coordinate(B) at ($(B3)-(dy)$);
\draw (A2)..controls +(0,1) and +(0,-1)..(A);
\draw [line width=1ex,white](B1)..controls +(0,1) and +(0,-1)..(B);
\draw (B1)..controls +(0,1) and +(0,-1)..(B);
\draw (A1)..controls +(0,1) and +(0,-1)..(A);
\draw (B2)..controls +(0,1) and +(0,-1)..(B);
\draw (B)--(B3);
\draw(A)--(A3);
\end{tikzpicture}
$$
As observed in \cite{qpl}, the fusion product is the semi-classical limit of the tensor product of monoids in the braided monoidal category $U\g\text{-mod}^\Phi$. A similar observation holds for quasi-Poisson algebras in any infinitesimally braided category.
\end{rem}

Slightly more generally, suppose that $\g$ has an invariant element $t_\g\in(S^2\g)^\g$ and that $\h$ is another Lie algebra with an element $t_\h\in(S^2\h)^\h$. For the Lie algebra $\g\oplus\g\oplus\h$ we use the element $t_{\g\oplus\g\oplus\h}=t_\g\oplus t_\g\oplus t_\h$ and for $\g\oplus\h$ we use $t_{\g\oplus\h}=t_\g\oplus t_\h$.

\begin{defn}[\cite{a-ks-m}]
If $(M,\rho,\pi)$ is a $\g\oplus\g\oplus\h$-quasi-Poisson manifold then the \emph{(internal) fusion} of $M$ is $M$ with the $\g\oplus\h$-quasi-Poisson structure given by
$$\rho_\circledast(u,v)=\rho(u,u,v)\quad(\forall u\in\g,v\in\h)$$
and
$$\pi_\circledast=\pi-\frac{1}{2}(\rho_1\wedge\rho_2)(t_\g)$$
where $\rho_1$ and $\rho_2$ are the actions of the first and of the second $\g$ in $\g\oplus\g\oplus\h$ respectively. 
\end{defn}
Again, $\sigma_\circledast$ is given by 
\begin{equation}\label{eq:fus}
\sigma_\circledast=\sigma+(\rho_2\otimes\rho_1)(t_\g).
\end{equation}

Fusion is compatible with central maps in the following way.
\begin{thm}\label{thm:fus}
Let
$$
\begin{tikzcd}[column sep=small]
& M_1 \arrow{dl}[swap]{\mu_L}\arrow{dr}{\mu_R} & \\
N & & N'
\end{tikzcd}
\begin{tikzcd}[column sep=small]
& M_2 \arrow{dl}[swap]{\nu_L}\arrow{dr}{\nu_R} & \\
N' & & N''
\end{tikzcd}
$$
be central pairs of $\g$-quasi-Poisson manifolds.
 If $\mu_R\times\nu_L$ is transverse to the diagonal $N'_{diag}\subset N'\times N'$ (in particular, if one of $\mu_R$, $\nu_L$ is a submersion) then the fibre product
$$M_1\times_{N'} M_2\subset M_1\times M_2$$
over $\mu_R$ and $\nu_L$ is a $\g$-quasi-Poisson submanifold of $M_1\circledast M_2$ and
$$
\begin{tikzcd}[column sep=small]
& M_1\times_{N'} M_2 \arrow{dl}[swap]{\mu_L}\arrow{dr}{\nu_R} & \\
N & & N''
\end{tikzcd}
$$
is a central pair.
\end{thm}
\begin{proof}
The submanifold $M_1\times_{N'} M_2\subset M_1\times M_2$ is $\g$-invariant. To show that it is a quasi-Poisson submanifold, we need to check that 
$$\sigma_\circledast(\mu_R^*\gamma-\nu_L^*\gamma,\cdot)=0$$
for every $(x,y)\in M_1\times_{N'} M_2\subset M_1\times M_2$ and for every $\gamma\in T^*_zN'$ where $z={\mu_R(x)=\nu_L(y)}$.
Centrality gives us
$$\sigma_2(\nu_L^*\gamma,\cdot)=0$$
$$\sigma_1(\mu_R^*\gamma,\cdot)=\sigma_1(\mu_R^*\gamma,\cdot)+\sigma_1(\cdot,\mu_R^*\gamma)=(\rho_2^{\otimes2}(t))(\cdot,\mu_R^*\gamma)=((\rho_2\otimes\rho_{N'})(t))(\gamma,\cdot).$$
Since 
$$((\rho_2\otimes\rho_1)(t))(\mu_R^*\gamma-\nu_L^*\gamma,\cdot)=((\rho_2\otimes\rho_1)(t))(-\nu_L^*\gamma,\cdot)=-((\rho_2\otimes\rho_{N'})(t))(\gamma,\cdot),$$
the expression \eqref{eq:fusion-si} for $\sigma_\circledast$ gives us
$$\sigma_\circledast(\mu_R^*\gamma-\nu_L^*\gamma,\cdot)=0,$$
as we needed.

To finish the proof that we have a central pair we need to check that $\mu_L$ and $\nu_R$ are left and right central (respectively) on $M_1\circledast M_2$. For $\alpha\in\Omega^1(N')$ we get $\sigma_\circledast(\mu_L^*\alpha,\cdot)=(\sigma_1+\sigma_2+(\rho_2\otimes\rho_1)(t))(\mu_L^*\alpha,\cdot)=0$ (the $\sigma_1$-term vanishes because $\mu_L:M_2\to N$ is left-central, the $\sigma_2$ and $\rho_2\otimes\rho_1$ terms are obviously zero), hence $\mu_L:M_1\circledast M_2\to N$ is indeed left-central. The proof of right centrality of $\nu_R$ is similar.
\end{proof}

\section{Quasi-Poisson structure on moduli spaces and its centers}\label{sec:moduli}

\subsection{The Poisson structure of Atiyah--Bott and Goldman}\label{sect:abg}
Let us first recall the Poisson structure of Atiyah--Bott \cite{a-b} and Goldman \cite{gold} on moduli spaces of flat connections on a  surface. Let $\g$ be a Lie algebra with an invariant element $t\in (S^2\g)^\g$, $G$ a connected Lie group with the Lie algebra $\g$, and $\Sigma$ an oriented compact surface, possibly with a boundary.

Let $P\to\Sigma$ be a principal $G$-bundle with a flat connection $A$ and $\g_P$ the associated adjoint vector bundle. The bundle $\g_P$ inherits the flat connection $A$, i.e.\ we can see it as a local system on $\Sigma$.  We shall denote this flat vector bundle by $\g_{P,A}$. 

The element $t:\g^*\times\g^*\to\R$ gives us a pairing $\g_{P,A}^*\times\g_{P,A}^*\to\R$, which in turn gives an intersection pairing on homology
\begin{equation}\label{eq:int_pair}
\pi:H_1(\Sigma;\g_{P,A}^*)\times H_1(\Sigma;\g_{P,A}^*)\to\R.
\end{equation}

 Let 
now
$$M_\Sigma(G)=\Hom(\pi_1(\Sigma),G)/G$$
be the moduli space of flat connections on principal $G$-bundles over $\Sigma$.
 The (formal) tangent space of $M_\Sigma(G)$ at $[P,A]$ is
\begin{subequations}
\begin{equation}
T_{[P,A]}M_\Sigma(G)=H^1(\Sigma;\g_{P,A}),
\end{equation}
the cotangent space is thus
\begin{equation}\label{eq:ctg}
T_{[P,A]}^*M_\Sigma(G)=H_1(\Sigma;\g_{P,A}^*),
\end{equation}
\end{subequations}
and the intersection pairing \eqref{eq:int_pair} becomes a bivector field on $M_\Sigma(G)$. It is the Poisson structure of Atiyah--Bott and Goldman.

When $\Sigma$ is closed and $t$ is non-degenerate then the pairing \eqref{eq:int_pair} is non-degenerate by Poincar\'e duality. The Poisson structure is symplectic in this case, the symplectic form $\omega$ is equal to the corresponding intersection pairing on $ H^1(\Sigma;\g_{P,A})$, and can thus be expressed in terms of 1-forms as
$$\omega([\alpha],[\beta])=\int_\Sigma\langle\alpha\wedge\beta\rangle\quad([\alpha],[\beta]\in H^1(\Sigma;\g_{P,A})),$$
where $\langle,\rangle$ is the pairing on $\g$ coming from $t$.

\subsection{Quasi-Poisson structures on moduli spaces}
The most important examples of quasi-Poisson manifolds are  moduli spaces of flat connections on a surface with marked points on the boundary \cite{a-m-m,a-ks-m}. Here we present these quasi-Poisson structures in terms of intersection pairing, as in Section \ref{sect:abg}. This point of view is significantly simpler than that of \emph{op.\ cit.} and reduces the problem of finding left and right leaves to Poincar\'e duality. Let us remark that, unlike \emph{op.\ cit.}, we don't discuss moment maps, as we replace them with central maps and pairs (see Section \ref{sect:mom} for the relations between moment maps and central pairs).

Let  $\Sigma$ be an oriented compact surface with boundary and let $V\subset\partial\Sigma$ be a finite set. For simplicity we suppose that $V$ meets every component of $\Sigma$. 

Let
$$M_{\Sigma, V}(G)=\Hom(\Pi_1(\Sigma,V),G)$$
where $\Pi_1(\Sigma,V)$ is the fundamental groupoid of $\Sigma$ with the base set $V$. It is the moduli space of flat connections on principal $G$-bundles over $\Sigma$, trivialized over $V$.

The set $M_{\Sigma, V}(G)$ is naturally a smooth manifold (it can be identified with $G^E$, where $E$ is the edge set of a graph $\Gamma$ embedded to $\Sigma$ with vertex set $V$, such that the embedding $\Gamma\subset\Sigma$ is a homotopy equivalence), with a natural action of $G^V$, given by 
$$(g\cdot f)(\gamma)=g^{\phantom{-1}}_{\on{head}(\gamma)}f(\gamma)g^{-1}_{\on{tail}(\gamma)}\quad
(f:\Pi_1(\Sigma,V)\to G,\gamma\in\Pi_1(\Sigma,V),g\in G^V),$$
or equivalently by changing the trivializations of the principal bundles over $V$.

Similar to \eqref{eq:ctg}, the cotangent space of $M_{\Sigma,V}(G)$ is the relative homology
\begin{equation}\label{eq:qctg}
T^*_{[P,A]}M_{\Sigma,V}(G)=H_1(\Sigma,V;\g^*_{P,A}).
\end{equation}
We can define an ``intersection pairing'' on $H_1(\Sigma,V;\g^*_{P,A})$ in the following way.
Let us   split $V$ to two disjoint subsets $V=V_+\sqcup V_-$. 
Let us move every point in $V_+$ a little along $\partial\Sigma$ in the direction given by the orientation induced from $\Sigma$, and every point in $V_-$ in the opposite direction. Let us denote the set of moved points by $\hat V$. Since $V$ and $\hat V$ are disjoint, we have a well-defined intersection pairing
$$H_1(\Sigma,V;\g^*_{P,A})\times H_1(\Sigma,\hat V;\g^*_{P,A})\to\mathbb R.$$
As $\hat V$ was obtained from $V$ by an isotopy, we have a natural isomorphism
$$H_1(\Sigma,V;\g^*_{P,A})\cong H_1(\Sigma,\hat V;\g^*_{P,A}).$$
Composing it with the intersection pairing we get the pairing
\begin{multline*}
\sigma_{V_+,V_-}:H_1(\Sigma,V;\g^*_{P,A})\times H_1(\Sigma,V;\g^*_{P,A}) \to H_1(\Sigma,V;\g^*_{P,A})\times H_1(\Sigma,\hat V;\g^*_{P,A})\to\mathbb R,
\end{multline*}
which can be viewed via \eqref{eq:qctg} as a tensor field on $M_{\Sigma,V}(G)$. (Perhaps in simpler terms: the intersection pairing on $H_1(\Sigma,V;\g^*_{P,A})$ is not well defined, as it is not clear how to count the intersections at $V$. The pairing $\sigma_{V_+,V_-}$ is defined by a particular rule saying which of these intersections are counted (those that survive the isotopy) and which are not (those that disappear).)

The skew-symmetric part of the pairing $\sigma_{V_+,V_-}$
\begin{equation}\label{eq:skews}
\pi=\frac{1}{2}\bigl(\sigma_{V_+,V_-}\!-\sigma_{V_+,V_-}\!{}^{op}\bigr)
\end{equation}
 is independent of the decomposition $V=V_+\sqcup V_-$; on cycles it is given by the rule that  intersections at the points of $V$ are counted with weight $1/2$.

The following theorem was essentially proven in \cite{m-t} and \cite{lb-se}, though using a somewhat different language of ``homotopy intersection pairing''. Intersection pairing with  the local system $\g_{P,A}^*$ seems to be better suited for our needs. Let $\bar\g=\g$, with $t$ replaced by $-t$.
\begin{thm}
The bivector field $\pi$ given by \eqref{eq:skews}
 defines a $\g^{V_+}\oplus\bar\g^{V_-}$-quasi-Poisson structure on $M_{\Sigma,V}(G)$. The tensor field $\sigma_{V_+,V_-}$ is the corresponding $\sigma$-tensor.
\end{thm}
\begin{proof}
This proof is adapted from \cite[Theorems 2 and 3]{lb-se}.

The theorem is valid if $(\Sigma, V_+\sqcup V_-)$ is a disjoint union of disks, with one point in $V_+$ and one point in $V_-$ on each disk. In this case $\sigma_{V_+,V_-}=0$ and $M_{\Sigma,V}(G)$ is a quasi-Poisson-commutative manifold.

If the theorem is valid for a surface $(\Sigma,V=V_+\sqcup V_-)$ then it is also valid for $(\Sigma',V'=V'_+\sqcup V'_-)$, for these $(\Sigma',V')$:
\begin{enumerate}
\item $\Sigma'=\Sigma$, $V'=V-\{x\}$ for some $x\in V$. In this case $M_{\Sigma', V'}(G)=M_{\Sigma, V}(G)/G$, and the quasi-Poisson structure on $M_{\Sigma, V}(G)$ reduces to the quasi-Poisson structure on $M_{\Sigma', V'}(G)$.
\item Let $x,y\in V_+$ and let $\Sigma'$ be obtained from $\Sigma$ by a ``corner connected sum'':
$$
\begin{tikzpicture}[scale=0.5]
\fill[color=white!90!black] (0,0.5) -- (2.5,1.5) -- (0,2.5);
\fill[color=white!90!black] (6,0.5) -- (3.5,1.5) -- (6,2.5);
\draw (0,0.5) -- (2.5,1.5) -- (0,2.5);
\draw (6,0.5) -- (3.5,1.5) -- (6,2.5);
\node[below] at (2.5,1.5) {$x$};
\node[below] at (3.5,1.5) {$y$};
\draw[->] (7.5,1.5) -- (9.5,1.5);
\fill[color=white!90!black]  (11,0.5) -- (13,1.5) -- (15,0.5)-- plot[smooth, tension=.7] coordinates {(15,2.5) (13,2) (11,2.5)};
\draw (11,0.5) -- (13,1.5) -- (15,0.5);
\draw  plot[smooth, tension=.7] coordinates {(15,2.5) (13,2) (11,2.5)};
\node[below] at (13,1.5) {$z$};
\node at (-1,1.5) {$\Sigma$};
\node at (16,1.5) {$\Sigma'$};
\end{tikzpicture}
$$
Let $V'$ be the image of $V$ in $\Sigma'$, i.e.\ with $x$ and $y$ identified (denoted $z$ on the picture). In this case $M_{\Sigma', V'}(G)\cong M_{\Sigma, V}(G)$ (the isomorphism is induced by the gluing map $\Sigma\to\Sigma'$) and the quasi-Poisson structure on $M_{\Sigma', V'}(G)$ is obtained from  the quasi-Poisson structure on $M_{\Sigma, V}(G)$ by fusion of the $\g$'s acting at $x$ and $y$. Indeed, the additional term $(\rho_2\otimes\rho_1)(t)$ in \eqref{eq:fus} corresponds to the new intersections close to the identified $x$ and $y$. The same works for $x,y\in V_-$ if we reverse the order of $x$ and $y$ in the corner connected sum.
\end{enumerate}

Using these operation we can get to arbitrary $(\Sigma,V_+\sqcup V_-)$.
\end{proof}

\subsection{Centers of moduli spaces}

The set $V\subset\partial\Sigma$ splits $\partial\Sigma$ to finitely many \emph{boundary arcs}. There may be boundary circles without any marked points; we shall call them \emph{uncut circles}. We shall say that a boundary arc from $x\in V$ to $y\in V$ (the arc is oriented by the  orientation induced from $\Sigma$) is \emph{left} if $x\in V_-$ and $y\in V_+$ (i.e.\ if the corresponding $\hat x,\hat y\in \hat V$ are ``outside'' of the arc) and \emph{right} if $x\in V_+$ and $y\in V_-$ (i.e\ if $\hat x,\hat y$ are ``inside''). Notice that the number of the left arcs is equal to the number of the right arcs; there may be arcs which are neither left nor right.
\begin{figure}[h]
\includegraphics[width=5cm]{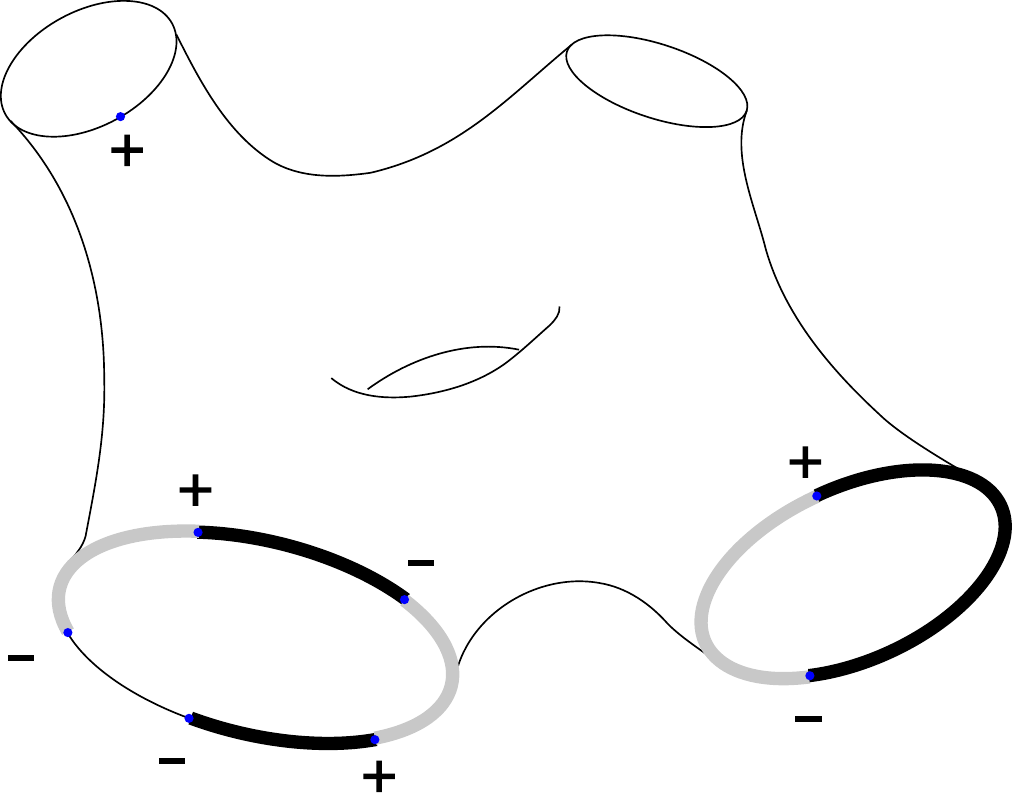}
\caption{Left boundary arcs are {\color{white!60!black}\bf{light}} and right are {\bf dark}; the remaining parts of the boundary are thin. Points in $V_+$ are marked by $+$ and in $V_-$ by $-$.}
\end{figure}

Let $L$ be the set of the left arcs. The holonomies along the left arcs give us a map
$$\mu_L:M_{\Sigma,V}(G)\to G^L,\ f\mapsto (f(a))_{a\in L}$$
and likewise, if $R$ is the set of the right arcs, we get a map
$$\mu_R:M_{\Sigma,V}(G)\to G^R,\ f\mapsto (f(a))_{a\in R}.$$
\begin{thm}
The map $\mu_L$ is left-central and the map $\mu_R$ is right-central.
\end{thm}
\begin{proof}
Let $[c]\in H_1(\Sigma,V;\g_{P,A}^*)$ be such that $c$ is supported on the left arcs. Then for any $[c']\in H_1(\Sigma,V;\g_{P,A}^*)$ we have $\sigma_{V_+,V_-}([c],[c'])=0$, as the supports of $c$ and of the appropriately deformed $c'$ don't intersect. This proves that $\mu_L$ is left-central. The proof for $\mu_R$ is similar.
\end{proof}

When $t\in(S^2\g)^\g$ is non-degenerate, we can describe the left and the right leaves of $M_{\Sigma,V}(G)$ explicitly.
\begin{thm}\label{thm:mod-leaves}
If $t$ is non-degenerate then
the left (right) leaves of $M_{\Sigma,V}(G)$ are obtained by fixing the holonomies along the left  (right) boundary arcs (i.e.\ the value of $\mu_{L(R)}$) and the conjugacy classes of the holonomies along the uncut boundary circles. The big leaves are obtained by fixing only the conjugacy classes along the uncut circles.
\end{thm}

\begin{proof}
By Poincar\'e duality the intersection pairing
$$ H_1(\Sigma,\partial\Sigma-\hat V;\g^*_{P,A})\times H_1(\Sigma,\hat V;\g^*_{P,A})\to\R$$
is non-degenerate. The left kernel of the pairing $\sigma_{V_+,V_-}$, i.e.\ the annihilator of $T^LM_{\Sigma,V}(G)$, is thus the kernel of the map
$$T^*M_{\Sigma,V}(G)=H_1(\Sigma,V;\g^*_{P,A})\to H_1(\Sigma,\partial\Sigma-\hat V;\g^*_{P,A}),$$
i.e.\ (using the long exact sequence for the triple $V\subset\partial\Sigma- \hat V\subset\Sigma$) the image of
\begin{equation}\label{eq:img}
H_1(\partial\Sigma-\hat V,V;\g^*_{P,A})\to H_1(\Sigma,V;\g^*_{P,A}).
\end{equation}
For any connected component $K\subset\Sigma-\hat V$ we have $H_1(K,K\cap V;\g^*_{P,A})=0$ unless $K$ is a left arc (when $K\cap V$ contains two points) or if $K$ is an uncut circle. The image of $H_1(\partial\Sigma-\hat V,V;\g^*_{P,A})=\bigoplus_K H_1(K,K\cap V;\g^*_{P,A})$ under the map \eqref{eq:img} is the annihilator of the tangent space of the submanifold given by fixing the left holonomies and the uncut conjugacy classes. The proof for right leaves is similar, and the big leaves can be found using  $T^{big}=T^L+\on{Im}\rho$.
\end{proof}

\subsection{The split-symplectic case}
Theorem \ref{thm:mod-leaves} enables us to single out the case when $M_{\Sigma,V}(G)$ is split-symplectic.
\begin{thm}\label{thm:mon-ndg}
If $t$ is non-degenerate, $V$ meets every component of $\partial\Sigma$, and if the points in $V_+$ and $V_-$ alternate along $\partial\Sigma$ (i.e.\ if every boundary arc is either left or right) then
$$
\begin{tikzcd}[column sep=small]
& M_{\Sigma,V_+\sqcup V_-}(G) \arrow{dl}[swap]{\mu_L}\arrow{dr}{\mu_R} & \\
G^L & & G^R
\end{tikzcd}
$$
is a split-symplectic central pair. The same is true if we allow uncut circles, and replace $M_{\Sigma,V_+\sqcup V_-}(G)$ with any of its big leaves (i.e.\ if we fix the conjugacy class for every uncut circle).
\end{thm}
\begin{proof}
We need to verify that the stabilizers of points in $G^L$ and $G^R$ are Lagrangian. We already know that they are coisotropic, and they are isomorphic to $\g^{V_+}$, so they must be Lagrangian for dimension reasons. 
\end{proof}

In the split-symplectic case we can somewhat simplify the formula for the quasi-Poisson structure, and also express it in terms of differential forms. Let $\mathcal L\subset\partial\Sigma$ be the union of the left arcs and let $\mathcal R\subset\partial\Sigma$ the union of the right arcs, so that $\mathcal L\cup\mathcal R=\partial\Sigma$ and $\mathcal L\cap\mathcal R=V$ (for simplicity we treat the case with no uncut circles). By Poincaré duality the intersection pairing
$$H_1(\Sigma,\mathcal L;\g^*_{P,A})\times H_1(\Sigma,\mathcal R;\g^*_{P,A})\to\mathbb R$$
is non-degenerate. If we compose it with the maps (coming from the inclusions $V\subset\mathcal L$, $V\subset\mathcal R$)
$$H_1(\Sigma,V;\g^*_{P,A})\to H_1(\Sigma,\mathcal L;\g^*_{P,A}),\ H_1(\Sigma,V;\g^*_{P,A})\to H_1(\Sigma,\mathcal R;\g^*_{P,A})$$
we get the pairing $\sigma$ on $T^*M=H_1(\Sigma,V;\g^*_{P,A})$ where $M=M_{\Sigma,V}(G)$.

In terms of cohomology, we have
$$T^LM=H^1(\Sigma,\mathcal L;\g_{P,A}),\ T^RM=H^1(\Sigma,\mathcal R;\g_{P,A})$$
included in
$$TM=H^1(\Sigma,V;\g_{P,A}).$$
The intersection pairing (non-degenerate by Poincaré duality)
$$H^1(\Sigma,\mathcal L;\g_{P,A})\times H^1(\Sigma,\mathcal R;\g_{P,A})\to\mathbb R$$
is then the pairing
$$\sigma^{-1}:T^LM\times T^RM\to\mathbb R$$
inverse to $\sigma$. We can express it in terms of differential forms: if we represent a cohomology class from $H^1(\Sigma,\mathcal L;\g_{P,A})$ by a 1-form  $\alpha\in\Omega^1(\Sigma,\g_{P,A})$, $\alpha|_{\mathcal L}=0$, and similarly a class from $H^1(\Sigma,\mathcal R;\g_{P,A})$ by $\beta\in\Omega^1(\Sigma,\g_{P,A})$,  $\beta|_{\mathcal R}=0$, then
$$\sigma^{-1}([\alpha],[\beta])=\int_\Sigma\langle\alpha\wedge\beta\rangle.$$

\subsection{Examples of reduced spaces}
Let us now discuss some simple examples of reductions of $M_{\Sigma,V}(G)$'s.
If $x\in V$ then $M_{\Sigma,V-\{x\}}(G)=M_{\Sigma,V}(G)/G$, where $G$ acts at $x$; the quasi-Poisson structure from $M_{\Sigma,V}(G)$ descends to the quasi-Poisson structure on $M_{\Sigma,V-\{x\}}(G)$. This is a reduction in the sense of
Proposition \ref{prop:pred} (or of Proposition \ref{prop:red} if $V=\{x\}$).

As a slightly more complicated example, let $x\in V_+$ and $y\in V_-$. Let us consider the Lie algebra $\g\oplus\bar\g$ acting at $x$ and $y$ and its diagonal subalgebra $\g_{diag}\subset\g\oplus\bar\g$, which is coisotropic. In this case $M_{\Sigma,V}(G)/G_{diag}=M_{\Sigma',V-\{x,y\}}(G)$, where $\Sigma'$ is obtained by joining $x$ and $y$:
$$
\begin{tikzpicture}[scale=0.5]
\fill[color=white!90!black] (0,0.5) -- (2.5,1.5) -- (0,2.5);
\fill[color=white!90!black] (6,0.5) -- (3.5,1.5) -- (6,2.5);
\draw (0,0.5) -- (2.5,1.5) -- (0,2.5);
\draw (6,0.5) -- (3.5,1.5) -- (6,2.5);
\node[below] at (2.5,1.5) {$x$};
\node[below] at (3.5,1.5) {$y$};
\draw[->] (7.5,1.5) -- (9.5,1.5);
\fill[color=white!90!black] plot[smooth, tension=.7] coordinates {(11,0.5) (13,1.3) (15,0.5)} --(15,2.5) plot[smooth, tension=.7] coordinates {(15,2.5) (13,1.7) (11,2.5)}--(11,0.5);
\draw plot[smooth, tension=.7] coordinates {(11,0.5) (13,1.3) (15,0.5)};
\draw  plot[smooth, tension=.7] coordinates {(15,2.5) (13,1.7) (11,2.5)};
\node at (-1,1.5) {$\Sigma$};
\node at (16,1.5) {$\Sigma'$};
\end{tikzpicture}
$$

As the simplest example of central reduction, let us suppose that $\Sigma$ has a single boundary circle and that $V_+=\{x\}$, $V_-=\{y\}$. The circle $\partial\Sigma$ is cut by $x$ and $y$ to a left and a right arc. We thus have a central pair
$$
\begin{tikzcd}[column sep=small]
& M_{\Sigma,\{x,y\}}(G) \arrow{dl}[swap]{\mu_L}\arrow{dr}{\mu_R} & \\
G & & G
\end{tikzcd}.
$$
Reduction by $\g_{diag}\subset\g\oplus\bar\g$ will replace $\Sigma$ by $\Sigma'$ with two boundary circles. Choice of a $\g_{diag}$ orbit in each $G$ in the central pair corresponds to a choice of a conjugacy class for each of the two circles.

As the final example, let $\h,\h^*,\mf l\subset\g$ be Lagrangian Lie subalgebras. Let us suppose that $\h\cap\h^*=0$, so that $\h,\h^*\subset\g$ is  a Manin triple, and thus $\h$ a Lie bialgebra. By a theorem of Drinfeld \cite{dr-hom}, $\mf l$ defines a Poisson structure on $H/H\cap L$ which makes it to a Poisson $H$-space (for simplicity we shall work with local groups so that we don't have to spell out closedness conditions), and in this way we get a classification of Poisson homogeneous $H$-spaces.

To obtain this Poisson homogeneous space by reduction, let $(\Sigma,V)$ be a triangle and let us reduce its moduli space by $\mf c:=\mf l\oplus\h^*\oplus\h\subset\g\oplus\g\oplus\bar\g$, as on the figure:
$$
\begin{tikzpicture}[baseline=-1cm]
\coordinate[label=left:{$+$}] (A) at (0,0) ;
\coordinate[label=right:{$+$}] (B) at (2,0) ;
\coordinate[label=above:{$-$}] (C) at (1,1.7);
\filldraw[fill=white!90!black] (A)--(B)--(C)--cycle;

\draw[->] (3,1)--(4,1);

\coordinate[label=left:{$\mf l$}] (AA) at (5,0) ;
\coordinate[label=right:{$\h^*$}] (BB) at (7,0) ;
\coordinate[label=above:{$\mf h$}] (CC) at (6,1.7);
\filldraw[fill=white!90!black] (AA)--(BB)--(CC)--cycle;
\end{tikzpicture}
$$
If we constrain the holonomy along the left edge (which defines a right-central map to $G$) to be in the $\mf c$-orbit passing through 1, we get the Poisson homogeneous space $H/H\cap L$.

There are many other examples connected with the world of Poisson-Lie groups which can be obtained by reduction of moduli spaces. Some of them were studied in \cite{lb-se,lb-se-big} using moment map reduction. As we shall see in the next section, such a reduction is a special case of central reduction.

\section{Moment maps via centers; fusion of $D/H$-moment maps}\label{sect:mom}
In this section we shall see that central maps and central reduction contains, as a special case, the theory of (quasi-)Hamiltonian spaces and their reduction. In particular, if $X\subset M$ is a left leaf of a split-symplectic $\mf d$-manifold and if $\h\subset\mf d$ is the stabilizer of $X$ then $X$ carries a natural $(\h,\mf d)$-quasi-symplectic structure and the map from $X$ to the (local) space of the right leaves of $M$ is  a (local) moment map.

This point of view also allows us to define the fusion of $D/H$-valued moment maps, which was so far missing.

Throughout this section $\mf d$ denotes a Lie algebra with an invariant non-degenerate symmetric pairing.

Let $\mf h\subset\mf d$ be a Manin pair, i.e.\ $\mf d$ is a Lie algebra with an invariant non-degenerate symmetric pairing, and $\h$ is its Lagrangian subalgebra. Let $\h^*\subset\mf d$ be a Lagrangian complement of $\mf h$. Equivalently, $\h$ is a quasi-Lie bialgebra, with $\delta_\h:\h\to\h\wedge\h$ and $\phi_\h\in\bigwedge^3\h$ given by the $\h^*$ and $\h$ components of  
$$\h^*\otimes \h^*\subset\mf d\otimes\mf d\xrightarrow{[,]}\mf d\cong\h\oplus\h^*.$$

\begin{defn}[\cite{a-ks}]\label{def:qPgen}
An \emph{$(\h,\mf d;\h^*)$-quasi-Poisson manifold} is an $\h$-manifold $M$ with a bivector field $\pi$ such that
\begin{align*}
 [\pi,\pi]/2&=\rho^{\otimes3}_M(\phi_\h) \\
[\rho_M(u),\pi]&=-\rho^{\otimes2}_M(\delta_\h(u))\quad(\forall u\in\h).
\end{align*}
\end{defn}
If $\h^*\subset\dd$ is a Lie subalgebra (i.e.\ if $\mf h,\mf h^*\subset\mf d$ is a Manin triple and $\h$ a Lie bialgebra) then $\phi_\h=0$ and $\pi$ is thus a Poisson structure.

For any $(\h,\mf d;\h^*)$-quasi-Poisson manifold $M$ the distribution given by the image of
$$a:\h\oplus T^*M\to TM,\ (u,\alpha)\mapsto \rho(u)+\pi(\alpha,\cdot)$$
is integrable, as $a$ is the anchor map of a Lie algebroid structure (see \cite{b-c-s}). Its integral leaves are the minimal $(\h,\mf d;\h^*)$-quasi-Poisson submanifolds of $M$, and are called the \emph{quasi-symplectic leaves} of $M$.  $M$ is \emph{quasi-symplectic} if it contains just one quasi-symplectic leaf.

Notice that Definition \ref{def:qpoiss} is a special case of Definition \ref{def:qPgen}.
If $\g$ is a Lie algebra with an invariant element $t\in S^2\g$, and if we suppose for simplicity that $t$ is non-degenerate, then $\g_\text{diag}\subset\g\oplus\bar\g$ is a Manin pair. A $\g$-quasi-Poisson structure is the same as a $(\g_\text{diag},\g\oplus\bar\g;\g_\text{antidiag})$-quasi-Poisson structure.

Definition \ref{def:qPgen} needs to be complemented by an explanation of what happens if we change the complement $\h^*\subset\mf d$, as $\h^*$ is understood as auxiliary data.  Lagrangian complements are in 1-1 correspondence with elements $\tau\in\bigwedge^2\h$ via
$$\tau\mapsto\h^*_\tau:=\{(\tau(\cdot,\alpha),\alpha);\alpha\in\h^*\}\subset\h\oplus\h^*\cong\mf d$$
(so that $\h^*_0=\h^*$). If we replace $\h^*$ by $\h^*_\tau$ then $\pi$ has to be replaced by $\pi-\rho^{\otimes2}_M(\tau)$. The element $\tau$ is called a \emph{twist}. See \cite{a-ks} for details.

Let us observe that any $\mf d$-quasi-Poisson manifold $M$ carries also a $(\h,\mf d;\h^*)$-quasi-Poisson structure. Indeed, let $e_i$ be a basis of $\h$ and $e^i$ the dual basis of the complement $\h^*\subset\mf d$. The twist
$$\tau_{\h,\h^*}=\frac{1}{2}\sum_i e_i\wedge e^i\in\textstyle\bigwedge^2\mf d$$
corresponds to the new complement $\h\oplus\h^*\subset\mf d\oplus\bar{\mf d}$ of $\mf d_\text{diag}\subset\mf d\oplus\bar{\mf d}$. With this new complement $\h$ is a  quasi-Lie sub-bialgebra\footnote{
If $\mf a$ is a quasi-Lie bialgebra then a Lie subalgebra $\mf b\subset\mf a$ is a quasi-Lie sub-bialgebra  if $\delta_{\mf a}$ restricted to $\mf b$ has values in $\bigwedge^2\mf b$ and  $\phi_{\mf a}\in\bigwedge^3\mf b\subset\bigwedge^3\mf a$. This makes $\mf b$ with $\delta_{\mf b}:=\delta_{\mf a}|_{\mf b}$ and $\phi_{\mf b}:=\phi_{\mf a}$ to a quasi-Lie bialgebra.
} of $\mf d$, hence $M$ with the action of $\h$ and with
\begin{equation}\label{eq:pi'}
\pi'=\pi-\rho^{\otimes 2}(\tau_{\h,\h^*})
\end{equation}
is $(\h,\mf d;\h^*)$-quasi-Poisson.

\begin{thm}\label{thm:qp-qp}
Let $(M,\rho,\pi)$ be a split-symplectic $\mf d$-manifold and let $X\subset M$ be a left leaf. Let $\h\subset\mf d$ be the stabilizer of $X$, and let $\h^*\subset\mf d$ be a Lagrangian complement of $\h$. Then the bivector field \eqref{eq:pi'}
is tangent to $X$ and $(X,\rho|_\h,\pi'|_X)$ is a $(\h,\mf d;\h^*)$-quasi-symplectic manifold.
\end{thm}
\begin{proof}
As $(M,\rho|_\h,\pi')$ is $(\h,\mf d;\h^*)$-quasi-Poisson, it remains to show that $\pi'$ is tangent to $X$ and that the image of
$$a:\h\oplus T^*X\to TX,\ (u,\alpha)\mapsto \rho(u)+\pi'(\alpha,\cdot)$$
is $TX$. As
$$t/2+\tau_{\h,\h^*}=\sum_i e_i\otimes e^i,$$
we get
$$\pi'=\sigma-\sum_i\rho(e_i)\otimes\rho(e^i),$$
which implies, for every $x\in X$ and $\alpha\in T^*_xM$,
$$\pi'(\cdot,\alpha)=\sigma(\cdot,\alpha)-\sum_i\alpha(\rho(e^i))\rho(e_i)\in T_x^R M+\rho_x(\h)=T_x^R M=T_xX.$$
This also shows that the image of $a$ is equal to the image of $\alpha\mapsto\sigma(\cdot,\alpha)$, i.e.\ to $T_xX$.
\end{proof}
 Let us now recall the definition of moment maps.
\begin{defn}[\cite{a-ks,lb-se-big}]
If $(M,\rho,\pi)$ is a $(\h,\mf d;\h^*)$-quasi-Poisson manifold and if
 $N$ is a $\mf d$-manifold with coisotropic stabilizers then a map $\mu:M\to N$ is a \emph{$N$-valued moment map} if it is $\h$-equivariant and if
$$(1\otimes\mu_*)(\pi)=-(\rho_M\otimes 1)(Z_N),$$ 
 where $Z_N\in\Gamma(\h\otimes TN)$ is given by
$$\langle\alpha, Z_N\rangle=\rho_N(\alpha)\quad\forall \alpha\in\h^*\subset\mf d$$
and $\rho_N$ is the action of $\mf d$ on $N$.
\end{defn}

The most important cases are $N=D/H$, where $H\subset D$ is the connected group integrating $\h$ (provided $H\subset D$ is closed, or working with local groups), and more generally $N=D/H'$, where $\h'\subset\mf d$ is another Lagrangian Lie subalgebra. When $\delta_\h=0$ and $\phi_\h=0$, i.e.\ when  $M$ has a $\h$-invariant Poisson structure, we have $D\cong H\ltimes\h^*$ and $D/H\cong\h^*$; in this case a $D/H$-valued moment map is simply a moment map in the usual sense.

\begin{rem}
Suppose that $t\in(S^2\g)^\g$ is nondegenerate, so that $D=G\times G$. We can identify $D/G$ (where $G\subset D$ is the diagonal) with $G$. If $M$ is a $\g$-quasi-Poisson manifold, a moment map $\mu:M\to D/G=G$ is called in \cite{a-m-m,a-ks-m} a group-valued moment map.

  In this case left (and right) leaves of $M$ coincide with the big leaves \cite[Theorem 3]{qhg}. 
In particular, if $M$ is quasi-symplectic then $\sigma$ is non-degenerate. One can show that
\begin{equation}\label{eq:qham}
\sigma^{-1}=\omega+\frac{1}{2}\mu^*s\in\Gamma((T^*)^{\otimes2}M)
\end{equation}
where $\omega\in\Omega^2(M)$ is the skew-symmetric part of $\sigma^{-1}$ and $s$ is the bi-invariant (pseudo-)Riemann metric on $G$ given by $t$. The 2-form $\omega$, together with the action of $\g$ and the map $\mu$, makes $M$ to a quasi-Hamiltonian space in the sense of \cite{a-m-m}. The relation between $\pi$ and $\omega$ given by \eqref{eq:qham} is somewhat cleaner than the (equivalent) relation given in \cite{a-ks-m}, as it clarifies the role of non-degeneracy conditions, which are simply the  invertibility of both sides of Equation \eqref{eq:qham}.

\end{rem}

Theorem \ref{thm:qp-qp} can now be complemented as follows.

\begin{thm}\label{thm:qp-qp-mm}
If, in the context of Theorem \ref{thm:qp-qp}, $\mu_R:M\to N$ is a right-central map then $\mu_R|_X:X\to N$ is a moment map.
\end{thm}
\begin{proof}
Right-centrality of $\mu_R$ means $(1\otimes(\mu_R)_*)(\sigma)=0$, hence
\begin{multline*}
\bigl(1\otimes(\mu_R)_*\bigr)(\pi')=-\bigl(1\otimes(\mu_R)_*\bigr)\Bigl(\sum_i\rho_M(e_i)\otimes\rho_M(e^i)\Bigr)=\\
=-\sum_i\rho_M(e_i)\otimes\rho_N(e^i)=-(1\otimes\rho_M)(Z_N).
\end{multline*}
This shows that $\mu_R:M\to N$ is a moment map for the $(\h,\mf d;\h^*)$-quasi-Poisson structure $\pi'=\pi-\rho^{\otimes2}(\tau_{\h,\h^*})$ on $M$, and thus $\mu_R|_X$ is a moment map, too.
\end{proof}

\begin{example}
Let us return to moduli spaces of flat connections. In the case when $V=V_+$ the map $M_{\Sigma,V_+\sqcup \emptyset}(G)\to G^n$ given by the holonomies along the boundary arcs (where $n=|V_+|$ is the number of boundary arcs) is a moment map. This observation made in \cite{a-m-m,a-ks-m} was the reason why the authors of \emph{op.\ cit.} started developing the theory of group-valued moment maps and quasi-Poisson geometry. Let us explain these moment maps using central pairs. For simplicity we suppose (as in \emph{op.\ cit.}) that $t$ is non-degenerate and that $V_+$ intersects every component of $\partial\Sigma$ (which implies that $M_{\Sigma,V_+\sqcup \emptyset}(G)$ is quasi-symplectic).

Let us choose a point on each of the $n$ boundary arcs and let $V_-$ be the set of these new points. Then  $M=M_{\Sigma,V_+\sqcup V_-}(G)$ is as in Theorem \ref{thm:mon-ndg}, i.e.
$$
\begin{tikzcd}[column sep=small]
& M \arrow{dl}[swap]{\mu_L}\arrow{dr}{\mu_R} & \\
G^L & & G^R
\end{tikzcd}
$$
is a split-symplectic central pair. This implies that $X:=\mu_L^{-1}(1)\subset M$ (as any fibre of $\mu_L$) is a left leaf. Since $X\subset M$ is given by the condition that the holonomies along the left boundary arcs are equal to $1$, we can identify $X$ (by contracting the left arcs) with $M_{\Sigma,V_+\sqcup \emptyset}(G)$.  The quasi-Poisson structure on $X$ given by Theorem \ref{thm:qp-qp} is equal to the original quasi-Poisson structure on $M_{\Sigma,V_+\sqcup \emptyset}(G)$.

Theorem \ref{thm:qp-qp-mm} now says that $\mu_R|_X:X\to G^R=G^n$ is a moment map, and it  is the original moment map $M_{\Sigma,V_+\sqcup \emptyset}(G)\to G^n$.
\end{example}

To make the link between centers and moment maps complete, we need to be more specific about the category we work in, namely choose one of these posibilities:
\begin{enumerate}
\item If $D$ denotes a connected group integrating $\mf d$ and $H\subset D$ the connected subgroup integrating $\h$, we need to suppose that $H$ is closed in $D$. Moreover we should consider  only $\mf d$- and $\h$-actions which integrate to $D$- and $H$-actions.
\item Without imposing any restrictions, we can work with local groups $D$ and $H$. In this case $D/H$ denotes any manifold with a transitive $\mf d$-action and with a chosen point $[1]\in D/H$ whose stabilizer is $\h\subset\mf d$. We then have to understand the results in the appropriate local form.
\end{enumerate}
 To stress that we now replace $\h$- and $\mf d$-actions with $H$- and $D$-actions, we shall use terminology ``$D$-quasi-Poisson manifolds'', ``$(H,\mf d;\h^*)$-quasi-Poisson manifolds'', etc.

\begin{thm}\label{thm:corr}
There is 1-1 correspondence between $D$-central pairs
\begin{equation}\label{eq:cp}
\begin{tikzcd}[column sep=small]
& M \arrow{dl}[swap]{\mu_L}\arrow{dr}{\mu_R} & \\
D/H & &N
\end{tikzcd}
\end{equation}
and $(H,\mf d;\h^*)$-quasi-Poisson manifolds $X$ with a moment map
\begin{equation}\label{eq:mm}
\mu:X\to N.
\end{equation}

The correspondence is: If the central pair \eqref{eq:cp} is given then $X=\mu_L^{-1}([1])$ and $\mu=\mu_R|_X$. If $X$ and $\mu$ are given, then $M$ is obtained from $X$ by inducing the $H$-action to a $D$-action, i.e.
\begin{equation}\label{eq:ind}
M=(D\times X)/H
\end{equation}
where $H$ acts on $D\times X$ via $h\cdot(d,x)=(dh^{-1},h\cdot x)$; the action of $D$ on $M$ is $d'\cdot[(d,x)]=[(d'd,x)]$. The central maps are
$$\mu_L([d,x])=[d]\in D/H,\ \mu_R([d,x])=d\cdot\mu(x)\in N.$$
The link between the bivector field $\pi$ on $M$ and $\pi'$ on $X=\mu_L^{-1}([1])$ is given by \eqref{eq:pi'}.
\end{thm}
\begin{proof}
As the action of $\mf d$ on $D/H$ is transitive, the map $\mu_L$ is a submersion, and thus $\mu_L^{-1}([1])\subset M$ is a submanifold.

If we forget about bivector fields and understand \eqref{eq:cp} and \eqref{eq:mm} as diagrams of $D$- and $H$-equivariant maps respectively, then their equivalence is simply the universal property of the induction \eqref{eq:ind} from $H$-action  to $D$-action.

The bivector fields can be treated as follows. If the central pair \eqref{eq:cp} is given then, as in the proof of Theorem \ref{thm:qp-qp}, \eqref{eq:pi'} makes $M$ to a $(\h,\mf d;\h^*)$-quasi-Poisson manifold, and it is tangent to $X$, hence $X\subset M$ is a $(\h,\mf d;\h^*)$-quasi-Poisson manifold. As in the proof of Theorem \ref{thm:qp-qp-mm}, $\mu_R:M\to N$ is a moment map, and thus restricts to a moment map $X\to N$.

For the other direction, suppose that the bivector field $\pi'$ on $X$ is given.
We define $\pi$ first as a section of $(TM)|_X$ via
$$\pi=\pi'+\rho_M^{\otimes2}(\tau_{\h,\h^*}).$$
The property
$$[\rho_X(u),\pi']=-\rho^{\otimes2}_X(\delta(u))\quad(\forall u\in\h)$$
ensures that $\pi$ is $H$-invariant. We can thus extend it (uniquely) to a $D$-invariant bivector field on $M$. The fact that $\pi'$ is $(\h,\mf d;\h^*)$-quasi-Poisson then implies that $[\pi,\pi]/2=\rho_M^{\otimes 3}(\phi_{\mf d})$ at the points of $X$; as $\pi$ is $D$-invariant, this relation is satisfied everywhere on $M$, i.e.\ $\pi$ is $\mf d$-quasi-Poisson. Left and right centrality of $\mu_L$ and $\mu_R$ (first at the points of $X\subset M$ and then on entire $M$ by $D$-invariance)  then follows easily from $\sigma=\pi'+\sum_i\rho(e_i)\otimes\rho(e^i)$
and from the fact that $\mu$ is a moment map.
\end{proof}

We can now explaint why moment map reduction, in its most general form given in \cite{lb-se-big}, is a special case of central reduction. Namely, if $C\subset D$ is a Lagrangian subgroup and $\mathcal O_N\subset N$ is a $C$-invariant submanifold, and if $\mu:X\to N$ is a moment map, then the reduction theorem of \textit{op.\ cit.} makes $X_{red}:=\mu^{-1}(\mathcal O_N)/C\cap H$ to a Poisson manifold; if $X$ is quasi-symplectic, the action on $N$ is transitive with Lagrangian stabilizers, and $\mathcal O_N$ is a $C$-orbit, then $X_{red}$ is symplectic. If we induce $X$ to a central pair \eqref{eq:cp} then this reduction is simply the central reduction of $M$ with $\mathcal O_N\subset N$ and $\mathcal O_{D/H}=C\cdot [1]\subset D/H$.

As another application, we can define fusion product of $D/H$-valued moment maps. If $\mu:X_1\to D/H$ and $\nu:X_2\to D/H$ are moment maps for $(H,\mf d;\h^*)$-quasi-Poisson manifolds $X_1$ and $X_2$, we induce them to central pairs
$$
\begin{tikzcd}[column sep=small]
& M_1 \arrow{dl}[swap]{\mu_L}\arrow{dr}{\mu_R} & \\
D/H & & D/H
\end{tikzcd}
\begin{tikzcd}[column sep=small]
& M_2 \arrow{dl}[swap]{\nu_L}\arrow{dr}{\nu_R} & \\
D/H & & D/H
\end{tikzcd}.
$$
By Theorem \ref{thm:fus}, 
$$
\begin{tikzcd}[column sep=small]
& M_1\times_{D/H}M_2 \arrow{dl}[swap]{\mu_L}\arrow{dr}{\nu_R} & \\
D/H & & D/H
\end{tikzcd}
$$
is a central pair, and we define, using our correspondence, the fusion of $X_1$ and $X_2$ as the $(H,\mf d;\h^*)$-quasi-Poisson manifold
$$X_1\circledast X_2:=\mu_L^{-1}([1])\subset M_1\times_{D/H}M_2.$$
In other words,
$$X_1\circledast X_2=X_1\times_{D/H}(D\times X_2)/H,$$
where the fibre product is taken over the maps $\mu:X_1\to N$ and $\nu_L:(D\times X_2)/H\to D/H$. The fact that $X_1\circledast X_2$ is not $X_1\times X_2$ is probably the reason why it was so far elusive.

Similarly, we can define the conjugate $\bar\mu:\bar X\to D/H$ of a $D/H$-valued moment map $\mu:X\to D/H$. Let
$$
\begin{tikzcd}[column sep=small]
& M=(D\times X)/H \arrow{dl}[swap]{\mu_L}\arrow{dr}{\mu_R} & \\
D/H & &D/H
\end{tikzcd}
$$
be the central pair corresponding to $\mu:X\to D/H$ by Theorem \ref{thm:corr}. Let $\bar M=M$ with $\pi$ replaced by $-\pi$. $\bar M$ is still a $D$-quasi-Poisson manifold, but $\mu_L,\mu_R:\bar M\to D/H$ are now right and left central respectively. As a result,
$$\bar X:=\mu_R^{-1}([1])\subset\bar M$$
is $(H,\mf d;\h^*)$-quasi-Poisson and $\bar\mu:=\mu_L|_{\bar X}$ is a moment map.

\begin{example}
$D/H$ is a commutative  $\dd$-quasi-Poisson manifold, and we have the central pair
$$
\begin{tikzcd}[column sep=small]
& D/H\circledast D/H \arrow{dl}[swap]{p_1}\arrow{dr}{p_2} & \\
D/H & &D/H
\end{tikzcd}
$$
where $p_{1,2}$ are the projections. Our correspondence makes $p_1^{-1}([1])=D/H$ to a $(H,\dd)$-quasi-Poisson manifold with the moment map $\on{id}:D/H\to D/H$. To avoid confusion with the $\dd$-quasi-Poisson $D/H$, let us denote this $(H,\dd)$-quasi-Poisson manifold by $(D/H)_{(\h)}$. It was discovered in \cite{a-ks}.
The central pair corresponding the fusion product of $(D/H)_{(\h)}$ with itself is readily seen to be
$$
\begin{tikzcd}[column sep=small]
& D/H\circledast D/H\circledast D/H \arrow{dl}[swap]{p_1}\arrow{dr}{p_3} & \\
D/H & &D/H
\end{tikzcd}
$$
where $p_{1,3}$ are the projections to the first and the third factor respectively.
\end{example}

\appendix

\section{A non-degeneracy lemma}
All the vector spaces in this section are over a field $K$, $\on{char}K\neq 2$, and finite-dimensional.

Let $U$ and $U'$ be vector spaces, and let $\la\cdot,\cdot\ra:U\times U'\to K$ be a non-degenerate pairing. Let us introduce the following bilinear forms on $U\oplus U'$:
$$(u\oplus\alpha,v\oplus\beta)=\la u,\beta\ra$$
$$(u\oplus\alpha,v\oplus\beta)_{sym}=\la u,\beta\ra+\la v,\alpha\ra$$
$$(u\oplus\alpha,v\oplus\beta)_{skew}=\la u,\beta\ra-\la v,\alpha\ra.$$
A subspace $L\subset U\oplus U'$ is \emph{$(\cdot,\cdot)_{sym}$-Lagrangian} if $(x,y)_{sym}=0$ for all $y\in L$ iff $x\in L$.
\begin{prop}\label{prop:ndeg-aux}
If $L\subset U\oplus U'$ is $(\cdot,\cdot)_{sym}$-Lagrangian then
$$\ker\bigl((\cdot,\cdot)_{skew}|_L\bigr)=(L\cap U)\oplus(L\cap U').$$
\end{prop}
\begin{proof}
Notice that $(\cdot,\cdot)_{skew}|_L=2(\cdot,\cdot)|_L$, as $(\cdot,\cdot)_{sym}|_L=0$. This shows that both $(L\cap U)$ and $(L\cap U')$ are in the kernel, as $U$ and $U'$ are the right and left kernel of $(\cdot,\cdot)$.

If, on the other hand, $u\oplus\alpha\in\ker((\cdot,\cdot)_{skew}|_L)=\ker((\cdot,\cdot)|_L)$ then, for every $x\in U\oplus U'$, $(u\oplus0,x)_{sym}=(u\oplus0,x)=(u\oplus\alpha,x)=0$. Since $L$ is Lagrangian, this implies $u\oplus0\in L\cap U$ and thus also $0\oplus\alpha\in L\cap U'$, hence $u\oplus\alpha\in(L\cap U)\oplus(L\cap U')$.
\end{proof}

\begin{prop}\label{prop:ndeg}
Let $V$ be a vector space with a bilinear pairing $\sigma:V\times V\to K$, and $W$ a vector space with a symmetric non-degenerate pairing $t:W\times W\to K$. Let $f:V\to W$ be a linear map such that 
$$\sigma(v,v)=\frac{1}{2}t\bigl(f(v),f(v)\bigr)$$
 for every $v\in V$. Let $V_L,V_R\subset V$ be the left and right kernels of $\sigma$. 

Suppose that $V_L\cap\ker f=0$ and that $f(V_L)\subset W$ is $t$-Lagrangian (i.e.\ not just $t$-isotropic). 

If $C\subset W$ is $t$-Lagrangian then the kernel of $\sigma|_{f^{-1}(C)}$ (which is a skew-symmetric form) is 
$$\ker\sigma|_{f^{-1}(C)}=V_L\cap f^{-1}(C)+V_R\cap f^{-1}(C).$$
\end{prop}

\begin{proof}
Let us consider the vector space
$$X=V/V_L\oplus V/V_R\oplus W$$
and the injective map
$$F:V\to X,\ v\mapsto([v],[v],f(v)).$$
On $X$ we have the non-degenerate symmetric pairing $(,)_{sym}\oplus(-t)$, where the pairing $(,)_{sym}$ on $V/V_L\oplus V/V_R$ comes from the pairing $\la [v],[v']\ra=\sigma(v,v')$, $[v]\in V/V_L$, $[v']\in V/V_R$. The image of $V$ is isotropic, and for dimension reasons it is Lagrangian.

Since the composition of Lagrangian relations is Lagrangian, the image of the map
$$F':f^{-1}(C)\to V/V_L\oplus V/V_R,\ v\mapsto([v],[v])$$
is Lagrangian, being the composition of $F(V)$ and $C$. The kernel of $F'$ is 
$$(V_L\cap f^{-1}(C))\cap(V_R\cap f^{-1}(C)).$$

Finally, since
$$\sigma|_{f^{-1}(C)}(x,y)=(F'(x),F'(y))_{skew},$$
the result follows from Proposition \ref{prop:ndeg-aux}, applied to the Lagrangian subspace $F'(f^{-1}(C))$ of $V/V_L\oplus V/V_R$.

\end{proof}

\end{document}